\documentclass[12pt]{amsart}
\usepackage{enumerate, amsmath, amsthm, amsfonts, amssymb, xy}
\input xy
\xyoption{all}
\usepackage[margin=1in]{geometry}
\usepackage[dvips]{graphicx,color}
\def\Cc{{\mathcal C}}
\def\Oc{{\mathcal O}}
\def\Pc{{\mathcal P}}

\def\bold{\bf}

\def\eb{{\bold e}}

\def\0b{{\bold 0}}
\def\Pc{{\mathcal P}}

\def\Hc{{\mathcal H}}

\def\Ac{{\mathcal A}}     

\def\Cc{{\mathcal C}}

\usepackage{bm}
\bmdefine{\Bzero}{0}
\bmdefine{\Bone}{1}
\def\Bone{{\bf 1}}
\def\RR{{\mathbb R}}
\def\ZZ{{\mathbb Z}}

\def\QQ{{\mathbb Q}}
\newtheorem{theorem}{Theorem}[section]
\newtheorem{proposition}[theorem]{Proposition}
\newtheorem{lemma}[theorem]{Lemma}
\newtheorem{definition}[theorem]{Definition}
\newtheorem{example}[theorem]{Example}
\newtheorem{rmk}[theorem]{Remark}
\newtheorem{corollary}[theorem]{Corollary}

\newtheorem{property}[theorem]{Property}

\newenvironment{remark}{\begin{rmk}\rm}{\end{rmk}}
\numberwithin{equation}{section}
\newcommand{\op}{\mathcal{O}(P)}
\renewcommand{\phi}{\varphi}
\renewcommand{\emptyset}{\varnothing}

\makeatletter
\def\Ddots{\mathinner{\mkern1mu\raise\p@
\vbox{\kern7\p@\hbox{.}}\mkern2mu
\raise4\p@\hbox{.}\mkern2mu\raise7\p@\hbox{.}\mkern1mu}}
\makeatother

\newcommand{\conv}{\operatorname{conv}}

\begin{document}
\title{Cutting convex polytopes by hyperplanes}
\author{Takayuki Hibi and Nan Li}
\address{Takayuki Hibi,
Department of Pure and Applied Mathematics,
Graduate School of Information Science and Technology,
Osaka University,
Toyonaka, Osaka 560-0043, Japan}
\email{hibi@math.sci.osaka-u.ac.jp}
\address{Nan Li,
Department of Mathematics,
Massachusetts Institute of Technology,
Cambridge, MA 02139, USA}
\email{nan@math.mit.edu}
\begin{abstract}
Cutting a polytope is a very 
natural way to produce new classes of interesting polytopes. Moreover, it has been very enlightening to
explore which algebraic and combinatorial properties of the orignial polytope are 
hereditary to its subpolytopes obtained by a cut. In this work, we put our attention to all the 
seperating hyperplanes for some given polytope (integral and convex) and study the existence and classification
of such hyperplanes. 

We prove the exitence of seperating hyperplanes 
for the order and chain polytopes for any finite posets that are not a single chain; prove there are no
such hyperplanes for any Birkhoff polytopes. Moreover, we give a complete seperating hyperplane 
classification for the unit cube and its subpolytopes obtained by one cut, together with 
some partial classification results for order and chain polytopes.
\end{abstract}
\maketitle 
\section*{Introduction}
Let $\Pc \subset \RR^{n}$ be a convex polytope of dimension $d$ 
and $\partial \Pc$ its boundary.
If $\Hc \subset \RR^{n}$ is a hyperplane,
then we write $\Hc^{(+)}$ and $\Hc^{(-)}$ for the closed half-spaces 
of $\RR^{d}$ with $\Hc^{(+)} \cap \Hc^{(-)} = \Hc$.
We say that {\em $\Hc$ cuts $\Pc$} if 
$\Hc \cap (\Pc \setminus \partial \Pc) \neq \emptyset$ 
and if each vertex of the convex polytopes $\Pc \cap \Hc^{(+)}$
and $\Pc \cap \Hc^{(-)}$ is a vertex of $\Pc$.
When $\Hc \cap (\Pc \setminus \partial \Pc) \neq \emptyset$,
it follows that $\Hc$ cuts $\Pc$ if and only if, for each edge 
$e = \conv(\{v,v'\})$ of $\Pc$, where $v$ and $v'$ are vertices of $\Pc$,
one has $\Hc \cap e \subset \{ v, v' \}$. Cutting a polytope is a very 
natural way to produce new classes of interesting polytopes. 
For example, the hypersimplices are obtained from cutting the unit cube
by hyperplanes of the form $x_1+\cdots+x_n=k,k+1$, for some integer $0\le k<n$, which
is a class of very interesting and well-studied polytopes (see for example \cite{Stan2}, \cite{LP2} and \cite{L}). A similar
class of interesting polytopes obtained from cutting permutahedrons and in general any graphical
zonotopes are studied in \cite{LP}. In general, it is a very interesting problem to
explore which algebraic and combinatorial properties of $\Pc$
are hereditary to $\Pc \cap \Hc^{(+)}$ and $\Pc \cap \Hc^{(-)}$.
For example, in \cite{HLZ} the study on separating hyperplanes of 
the edge polytope $\Pc_{G}$ of a finite connected simple graph $G$
is achieved and it is shown that $\Pc_{G}$ is normal if and only if
each of $\Pc_{G} \cap \Hc^{(+)}$ and $\Pc_{G} \cap \Hc^{(-)}$
is normal.

In this paper, we look at the problem from another perspective, focusing more on the
hyerplane that cuts the polytope. We are interested in the exitence and
classification of such hyperplanes. Let us make it more
precise what we mean by a ``cut''.

If $\Hc$ cuts $\Pc$, then we call
$\Hc$ a {\em separating hyperplane} of $\Pc$.
If $\Hc$ is a separating hyperplane of $\Pc$, then
the {\em decomposition} of $\Pc$ via  
$\Hc$ is
\[
\Pc = (\Pc \cap \Hc^{(+)}) \cup (\Pc \cap \Hc^{(-)}).
\]
For example, if $[0,1]^3 \subset \RR^{3}$ is the 
unit cube,
then the hyperplane $\Hc \subset \RR^{3}$
define by the equation $x_{i} + x_{j} = 1$
with $1 \leq i < j \leq 3$ is a separating hyperplane
of $[0,1]^3$.

Unless $n = d$, two different separating hyperplanes 
$\Hc$ and $\Hc'$ of $\Pc$ might yield 
the same decomposition of $\Pc$. 
For example, if $[0,1]^2 \subset \RR^{3}$ is the square,
then its separating hyperplane defined by $x_{1} + x_{2} + x_{3} = 1$ and
that defined by $x_{1} + x_{2} - x_{3}= 1$ clearly yield
the same decomposition of $[0,1]^2$. 
 
An {\em integral} convex polytope is a convex polytope
any of whose vertices has the integer coordinates.
Let $\Pc \subset \RR^{n}$ be an integral convex polytope 
of dimension $d$ and 
suppose that $\partial \Pc \cap \ZZ^{n}$ is
the set of vertices of $\Pc$.
It then follows that
a hyperplane $\Hc \subset \RR^{n}$
is a separating hyperplane of $\Pc$ if and only if
each of the subpolytopes 
$\Pc \cap \Hc^{(+)}$ and $\Pc \cap \Hc^{(-)}$
is integral of dimension $d$.

The study of existence and classification for any general interal convex poltyopes can be very hard. In the present paper, 
we focus our study on the following classes of polytopes: 
the unit cube and its subpolytopes cut by one hyperplane, order and chain polytopes, and Birkhoff polytopes. We prove the exitence of seperating hyperplanes 
for the order and chain polytopes for any finite posets that are not a single chain (Theorem \ref{SepHypOC}), and prove there are no
 seperating hyperplanes for any Birkhoff polytopes (Theorem \ref{more}). Moreover, we give a complete seperating hyperplane 
classification for the unit cube and its subpolytopes cut by one hyperplane (Section 1), together with 
partial classification results for order and chain polytopes (Section 2).
%
%
\section{The unit cube}
Let $[0,1]^d \subset \RR^{d}$ be the unit cube with $d \geq 2$.
In the study of its separating hyperplane $\Hc$
it is assumed that $\Hc$ passes through the origin of $\RR^{d}$.
First of all, 
we discuss the question when a hyperplane $\Hc$ of $\RR^{d}$
passing through the origin 
\begin{eqnarray}
\label{hyperplane}
\Hc:\,a_1x_1+\cdots+a_dx_d=0,
\end{eqnarray}
where each $a_{i} \in \QQ$,
is a separating hyperplane of $[0,1]^d$.

\begin{lemma}
\label{hyone}
A hyperplane {\rm (\ref{hyperplane})}
is a separating hyperplane of $[0,1]^d$ 
if and only if
there exists $p$ and $q$ with $a_{p} > 0$ and $a_{q} < 0$
and 
all nonzero coefficients of $\Hc$ 
have the same absolute value.
\end{lemma}

\begin{proof}
{\bf (``If'')}
Let $e$ be an edge of $[0,1]^d$.  Then
\[
e = 
\{(x_1,\dots,x_d) \, ; \,
0\le x_i\le 1
\text{ and }
x_{j} = \varepsilon_{j}
\text{ for all $j \neq i$} \},
\]
where each $\varepsilon_{j} \in \{ 0, 1 \}$.
Suppose that
there exists $p$ and $q$ with $a_{p} > 0$ and $a_{q} < 0$
and that
all nonzero coefficients of $\Hc$ 
have the same absolute value.
Then we may assume that
\[
\Hc : x_{1} + \cdots + x_{s} - x_{s+1} - \cdots - x_{s+t} = 0,
\]
where $s > 0$, $t > 0$ and $s + t \leq d$.
If $i > s + t$, then either $\Hc \cap e = e$ or 
$\Hc \cap e = \emptyset$.
If $i \leq s + t$, then
$\Hc \cap e \subset \ZZ^{d}$.
Thus each of
$[0,1]^d \cap \Hc^{(+)}$ and $[0,1]^d \cap \Hc^{(-)}$
is integral. 
Moreover, since $s > 0$ and $t > 0$,
it follows that $\Hc \cap (0,1)^d \neq \emptyset$. 
Hence $\Hc$ is a separating hyperplane of $[0,1]^d$.

{\bf (``Only if'')}
If every coefficient $a_{i}$ of (\ref{hyperplane})
is nonnegative, then
$\Hc \cap [0,1]^d$ consists only of the origin.  Hence
$\Hc$ cannot be a separating hyperplane of $[0,1]^d$.
Thus there exists $p$ and $q$ with $a_{p} > 0$ and $a_{q} < 0$.

Now, suppose that 
there exist $i \neq j$ with
$a_i \neq 0$, $a_j \neq 0$ and $|a_i| \neq |a_j|$. 
Let, say, $|a_i| < |a_j|$.
Let $e$ is the edge defined by
$x_{i} = 1$ and $x_{k} = 0$ for all $k$ with $k \not\in \{ i, j\}$.
If $a_{i}a_{j} < 0$, then $0 < - a_{i}/a_{j} < 1$ and
$v = (v_{1}, \ldots, v_{n}) \in e$
with $v_{j} = - a_{i}/a_{j}$ belongs to $\Hc$.
Thus $\Hc$ cannot be a separating hyperplane of $[0,1]^d$.
Hence $a_{i}a_{j} > 0$.
In particular $|a_{p}| = |a_{q}|$.
Let $1 \leq k \leq d$ with
$a_{k} > 0$.  Then, since $a_{k}a_{q} < 0$,
it follows that $|a_{k}| = |a_{q}|$.
Similarly if
$a_{k} < 0$, then $|a_{k}| = |a_{p}|$.
Consequently, 
$|a_{k}| = |a_{p}| \, ( = |a_{p}|)$ 
for all $k$ with $a_{k} \neq 0$, as desired.
\end{proof}

Now, by virtue of Lemma \ref{hyone}, 
it follows that
a separating hyperplane of $[0,1]^d$ passing through the origin
is of the form
\begin{eqnarray}
\label{aaaaa}
\Hc : \, x_1+\dots+x_s-x_{s+1}-\dots-x_{s+t}=0
\end{eqnarray}
with $s > 0$, $t > 0$ and $s+t\le d$.
Moreover, in (\ref{aaaaa}), 
by replacing
$x_{s+i}$ with $1-x_{s+i}$ for
$1 \leq i \leq t$,
we can work with a separating hyperplanes of $[0,1]^d$
of the form
\begin{eqnarray}
\label{bbbbb}
\Hc : x_1+\dots+x_s+x_{s+1}+\dots+x_{s+t}=t.
\end{eqnarray}
Finally,
the equation (\ref{bbbbb}) can be rewritten as
\begin{eqnarray}
\label{ccccc}
\Hc : x_1+\dots+x_{k}=\ell,
\, \, \, \, \,
2\le k\le d, \, \, 1\leq\ell<k.
\end{eqnarray}

If $\Hc$ is a separating hyperplane (\ref{ccccc})
of $[0,1]^d$, then
\[
[0,1]^d \cap \Hc^{(+)}
=\{(x_1,\dots,x_d)\in [0,1]^d \, : \, x_1+\cdots+x_k\le
\ell\}, 
\]
\[ 
[0,1]^d \cap \Hc^{(-)} 
=\{(x_1,\dots,x_d)\in [0,1]^d \, : \, x_1+\cdots+x_k\ge
\ell\}.
\]
In $[0,1]^d \cap \Hc^{(-)}$, 
again by replacing 
$x_{i}$ with $1 - x_{i}$ for $1 \leq i \leq t$,
it follows that, since $0 < k - \ell < k$,
each of the subpolytopes 
$[0,1]^d \cap \Hc^{(\pm)}$ 
is, up to unimodular equivalence, of the form 
\begin{equation}
\label{p}
\{(x_1,\dots,x_d)\in [0,1]^d \, : \, 
x_1+\cdots+x_k\le \ell\}, 
\, \, \, \, \, 
2\le k\le d, \, \, 1\leq\ell< k.
\end{equation}

\begin{corollary}
\label{thenumberof}
The number of convex polytopes of the form 
$[0,1]^d \cap \Hc^{(\pm)}$, where $\Hc$ is a separating hyperplane
of $[0,1]^d$ is, up to unimodular equivalence, $d(d-1)/2$.
\end{corollary}

We now turn to the problem of finding a separating hyperplane
of (\ref{p}).
We say that a separating hyperplane of (\ref{p})
is a {\em second separating hyperplane of $[0,1]^d$
following} (\ref{ccccc}). 

\begin{lemma}
\label{general}
Each of the separating hyperplanes of $[0,1]^d$ is of the form
\[
\sum_{i \in I} x_{i} - \sum_{j \in J} x_{j} = h, 
\]
where $\emptyset \neq I \subset [\,d\,]$,
$\emptyset \neq J \subset [\,d\,]$,
$I \cap J = \emptyset$
and where $h \geq 0$ is an integer
with $0 \leq h < \sharp(I)$. 
\end{lemma}

\begin{proof}
Let $v = (v_{1}, \ldots, v_{d})$ be a vertex of $[0,1]^d$. 
Let $I \subset [\,d\,]$ and $J \subset [\,d\,]$ 
with $I \cup J = [\,d\,]$ and $I \cap J = \emptyset$
such that $v_{i} = 0$ if $i \in I$
and $v_{j} = 1$ if $j \in J$. 
Let 
\begin{eqnarray}
\label{hypergeneral}
\Hc : a_{1} x_{1} + \cdots + a_{d}x_{d} = \sum_{j \in J} a_{j},
\end{eqnarray}
with each $a_{i} \in \QQ$,
be a separating hyperplane of $[0,1]^d$ passing through $v$.
In (\ref{hypergeneral})
replace $x_{j}$ with $1 - x_{j}$ for $j \in J$, and the hyperplane
\begin{eqnarray}
\label{hypergeneralyyyyy}
\Hc' \, : \, 
\sum_{i \in I} a_{i}x_{i} + \sum_{j \in J} (- a_{j})x_{j} = 0
\end{eqnarray}
is a separating hyperplane of $[0,1]^d$ passing through the origin.
It then follows from Lemma \ref{hyone} 
that all nonzero coefficients 
of (\ref{hypergeneralyyyyy}) have the same absolute value.
Thus each of $a_{i}$'s and $a_{j}$' belongs to $\{ 0, \pm 1 \}$. 
It turns out that the equation (\ref{hypergeneral}) is 
\[
\Hc \, : \, \sum_{p \in I'} x_{p} - \sum_{q \in J'} x_{q} = h
\]
where $\emptyset \neq I' \subset [\,d\,]$, 
$\emptyset \neq J' \subset [\,d\,]$, 
$I \cap J = \emptyset$
and where $h \geq 0$ is an integer.
If $h \geq \sharp(I')$, then $[0,1]^d \subset \Hc^{(+)}$
or $[0,1]^d \subset \Hc^{(-)}$.
Hence $0 \leq h < \sharp(I')$.  If $0 \leq h < \sharp(I')$,
then 
\[
(\Hc^{(+)} \setminus \Hc) \cap [0,1]^d \neq \emptyset,
\, \, \, \, \, 
(\Hc^{(-)} \setminus \Hc) \cap [0,1]^d \neq \emptyset.
\] 
Thus $\Hc$ is, in fact, a separating hyperplane of $[0,1]^d$. 
\end{proof}

Let $\Hc' \subset \RR^{d}$ 
be a second separating hyperplane
of $[0,1]^d$ following (\ref{ccccc}).  
Clearly $\Hc'$ is a separating hyperplane of $[0,1]^d$.
It then follows from Theorem \ref{general} that
\begin{eqnarray}
\label{ddddd}
\Hc' \, : \, \sum_{i \in I} x_{i} - \sum_{j \in J} x_{j} = h, 
\end{eqnarray}
where $\emptyset \neq I \subset [\,d\,]$,
$\emptyset \neq J \subset [\,d\,]$,
$I \cap J = \emptyset$
and where $h \geq 0$ is an integer
with $0 \leq h < \sharp(I)$. 

\begin{theorem}
\label{second}
A 
hyperplane $\Hc'$ of {\rm (\ref{ddddd})} 
is a second separating hyperplane of $[0,1]^d$ following {\rm (\ref{ccccc})}
if and only if one of the following conditions 
is satisfied\hspace{0.25mm}{\rm : }
\begin{itemize}
\item $\sharp(J)+h+k-\sharp(X)\le \ell$\hspace{0.2mm}{\rm ; }
\item $\sharp(I)-h+k-\sharp(Y)\le \ell$\hspace{0.2mm}{\rm , }
\end{itemize}
where 
$X=I\cap[\,k\,]$ and $Y=J\cap[\,k\,]$.
\end{theorem}

\begin{proof}
Let $\Pc \subset \RR^{d}$ denote the subpolytope 
(\ref{p}) of $[0,1]^d$.
Then a hyperplane $\Hc'$ of {\rm (\ref{ddddd})} 
is a second separating hyperplane of $[0,1]^d$
following {\rm (\ref{ccccc})}
if and only if one has
$\Hc' \cap [0,1]^d \subset \Pc$.

{\bf (``If'')}
Let $v = (v_{1}, \ldots, v_{d}) \in [0,1]^d$ 
belong to $\Hc'$, i.e.,
$\sum_{i \in I} v_{i} - \sum_{j \in J} v_{j} = h$.
If $\sharp(J)+h+k-\sharp(X)\le \ell$, then
\begin{eqnarray*}
v_{1} + \cdots + v_{k}
& = & \sum_{i \in X} v_{i} 
+ \sum_{i \in [\,k\,] \setminus X} v_{i} \\
& \leq & \sum_{i \in I} v_{i} 
+ \sum_{i \in [\,k\,] \setminus X} v_{i} \\ 
& = & h + \sum_{j \in J} v_{j} + \sum_{i \in [\,k\,] \setminus X} v_{i} \\
& \leq & h + \sharp(J) + k - \sharp(X) \leq \ell.
\end{eqnarray*}
Hence $v \in \Pc$.
If $\sharp(I)-h+k-\sharp(Y)\le \ell$, then
\begin{eqnarray*}
v_{1} + \cdots + v_{k}
& = & \sum_{i \in Y} v_{i} 
+ \sum_{i \in [\,k\,] \setminus Y} v_{i} \\
& \leq & \sum_{i \in J} v_{i} 
+ \sum_{i \in [\,k\,] \setminus Y} v_{i} \\ 
& = & h + \sum_{j \in I} v_{j} + \sum_{i \in [\,k\,] \setminus Y} v_{i} \\
& \leq & h + \sharp(I) + k - \sharp(Y) \leq \ell.
\end{eqnarray*}
Hence $v \in \Pc$.

{\bf (``Only if'')}
Let 
$\sharp(J)+h+k-\sharp(X) > \ell$ and 
$\sharp(I)-h+k-\sharp(Y) > \ell$.
We claim the existence of
$v = (v_{1}, \ldots, v_{d}) \in [0,1]^d$
with $v \in \Hc'$ such that
$v_{1} + \cdots + v_{k} > \ell$.

Let $\sharp(I) \le \sharp(J)+h$.
Then $0 \leq h < \sharp(I) \le \sharp(J)+h$.
Thus there is
$v \in [0,1]^d$ belonging to $\Hc'$
with $v_{i} = 1$ for all $i \in I$
such that if
$j \in Y$ and $j' \in J \setminus Y$
then $v_{j} \geq v_{j'}$. 
Such $v \in \Hc' \cap [0,1]^d$ can be chosen with $v_{i} = 1$
for all $[\,k\,] \setminus (X \cup Y)$.
Then
\begin{eqnarray*}
v_{1} + \cdots + v_{k}
& \geq & \sharp(X) + \min\{\sharp(I) - h, \sharp(Y)\} +   
\sharp([\,k\,] \setminus (X \cup Y)) \\
& = & \sharp(X) + \min\{\sharp(I) - h, \sharp(Y)\} 
+ k - \sharp(X) - \sharp(Y) \\
& = & \min\{\sharp(I) - h, \sharp(Y)\} 
+ k - \sharp(Y) \\
& = & \min\{\sharp(I) - h + k - \sharp(Y), k\}.
\end{eqnarray*}
Since $\sharp(I) - h + k - \sharp(Y) > \ell$
and $k > \ell$, it follows that $v_{1} + \cdots + v_{k} > \ell$.

Let $\sharp(I) > \sharp(J)+h$.
Then there is
$v \in [0,1]^d$ belonging to $\Hc'$
with $v_{j} = 1$ for all $j \in J$
such that if
$i \in X$ and $i' \in I \setminus X$
then $v_{i} \geq v_{i'}$. 
Such $v \in \Hc' \cap [0,1]^d$ can be chosen with $v_{i} = 1$
for all $[\,k\,] \setminus (X \cup Y)$.
Then
\begin{eqnarray*}
v_{1} + \cdots + v_{k}
& \geq &
\sharp(Y) + \min\{\sharp(J) + h, \sharp(X)\} +   
\sharp([\,k\,] \setminus (X \cup Y)) \\
& = & \sharp(Y) + \min\{\sharp(J) + h, \sharp(X)\} 
+ k - \sharp(X) - \sharp(Y) \\
& = & \min\{\sharp(J) + h, \sharp(X)\} 
+ k - \sharp(X) \\
& = & \min\{\sharp(J) + h + k - \sharp(X), k\}.
\end{eqnarray*}
Since $\sharp(J) + h + k - \sharp(X) > \ell$
and $k > \ell$, it follows that $v_{1} + \cdots + v_{k} > \ell$.
\end{proof}

\begin{corollary}
\label{simple}
Let
\begin{eqnarray}
\label{fffff}
x_{1} + \cdots + x_{d} = \ell, \, \, \, \, \, 1 \leq \ell < d
\end{eqnarray}
be a separating hyperplane of $[0, 1]^{d}$.  Then
a hyperplane
\[
x_{1} + \cdots + x_{s} = x_{s+1} + \cdots + x_{s+t} + h, \, \, \, \, \, 0 \leq h < s 
\]
is a second separating hyperplane of $[0, 1]^{d}$ 
following {\rm (\ref{fffff})} if and only if 
one of the following conditions is satisfied\hspace{0.25mm}{\rm : }
\begin{itemize}
\item $d - \ell \leq s - (t + h)$\hspace{0.2mm}{\rm ; }
\item $d - \ell \leq (t + h) -s$\hspace{0.2mm}{\rm . }
\end{itemize}
\end{corollary}

%
%
\section{Order and chain polytopes}
Let $P = \{x_{1}, \ldots, x_{d}\}$ be a finite 
partially ordered set (\,{\em poset} for short).
To each subset $W \subset P$, we associate $\rho(W) = \sum_{i \in W}\eb_{i} \in \RR^{d}$,
where $\eb_{1}, \ldots, \eb_{d}$ are the unit coordinate vectors of $\RR^{d}$.
In particular $\rho(\emptyset)$ is the origin of $\RR^{d}$.
A {\em poset ideal} of $P$ is a subset $I$ of $P$ such that,
for all $x_{i}$ and $x_{j}$ with
$x_{i} \in I$ and $x_{j} \leq x_{i}$, one has $x_{j} \in I$.
An {\em antichain} of $P$ is a subset
$A$ of $P$ such that $x_{i}$ and $x_{j}$ belonging to $A$ with $i \neq j$ are incomparable.
We say that $x_{j}$ {\em covers} $x_{i}$ if $x_{i} < x_{j}$ and
$x_{i} < x_{k} < x_{j}$ for no $x_{k} \in P$.
A chain $x_{j_{1}} < x_{j_{2}} < \cdots < x_{j_{\ell}}$ of $P$ is called
{\em saturated} if $x_{j_{q}}$ covers $x_{j_{q-1}}$ for $1 < q \leq \ell$. A {\em maximal chain} is a
saturated chain such that $x_{j_{1}}$ is a minimal element and $x_{j_{\ell}}$ is a maximal element of the poset. 

The {\em order polytope} of $P$ is the convex polytope $\Oc(P) \subset \RR^{d}$
which consists of those $(a_{1}, \ldots, a_{d}) \in \RR^{d}$ such that
$0 \leq a_{i} \leq 1$ for every $1 \leq i \leq d$ together with
\[
a_{i} \geq a_{j}
\]
if $x_{i} \leq x_{j}$ in $P$.

The {\em chain polytope} of $P$ is the convex polytope $\Cc(P) \subset \RR^{d}$
which consists of those $(a_{1}, \ldots, a_{d}) \in \RR^{d}$ such that
$a_{i} \geq 0$ for every $1 \leq i \leq d$ together with
\[
a_{i_{1}} + a_{i_{2}} + \cdots + a_{i_{k}} \leq 1
\]
for every maximal chain $x_{i_{1}} < x_{i_{2}} < \cdots < x_{i_{k}}$ of $P$.

One has $\dim \Oc(P) = \dim \Cc(P) = d$. The number of vertices of $\Oc(P)$ is equal to that of $\Cc(P)$.
Moreover, the volume of $\Oc(P)$ and that of $\Cc(P)$ are equal to $e(P)/d!$, where
$e(P)$ is the number of linear extensions of $P$ (\cite[Corollary 4.2]{Stanley}). It also follows from \cite{Stanley}
that the facets of $\Oc(P)$ are the following:
\begin{itemize}
\item
$x_i = 0$, where $x_i \in P$ is maximal;
\item
$x_j = 1$, where $x_j \in P$ is minimal;
\item
$x_{i} = x_{j}$, where $x_j$ covers $x_i$,
\end{itemize}
and that the facets of $\Cc(P)$ are the following:
\begin{itemize}
\item
$x_i = 0$ for all $x_i \in P$;
\item
$x_{i_{1}} + \cdots + x_{i_{k}} = 1$,
where $x_{i_{1}} < \cdots < x_{i_{k}}$
is a maximal chain of $P$.
\end{itemize} Moreover, we have the following descriptions for vertices, which will be used frequently in this section.
\begin{lemma}[\cite{Stanley}]\label{vertex}
 \begin{enumerate}
  \item Each vertex of $\Oc(P)$ is
 $\rho(I)$ such that $I$ is a poset ideal of $P$;
  \item each vertex of $\Cc(P)$ is
 $\rho(A)$ such that $A$ is an antichain of $P$.
 \end{enumerate}

\end{lemma}

\subsection{Existence of seperating hyerplanes for order and chain polytopes}

In this subsection, we study the existence of separating hyperplanes of order polytopes and chain polytopes 
(see Theorem \ref{SepHypOC}). First we need 
an explicit description of edges in terms of vertices.

\begin{lemma}
\label{orderedge}
Let $I$ and $J$ be poset ideals of $P$ with $I \neq J$.
Then
$\conv(\{ \rho(I), \rho(J) \})$ forms an edge of $\op$ if and only if 
$I\subset J$ and $J \backslash I$ is connected in $P$.
\end{lemma}

\begin{proof}
If there exists a maximal element $x_{i}$ of $P$ not belonging to
$I \cup J$, then
$\conv(\{ \rho(I), \rho(J) \})$ lies in the facet  
$x_{i} = 0$.  
If there exists a minimal element $x_{j}$ of $P$ belonging to
$I \cap J$, then $\conv(\{ \rho(I), \rho(J) \})$ 
lies in the facet  
$x_{j} = 1$.
Hence, working with induction on $d$, 
we may assume that $I \cup J = P$ and $I \cap J = \emptyset$.

Let neither $I = \emptyset$ nor $J = \emptyset$.
Then $P$ is the disjoint union
of $I$ and $J$.  Now, suppose that $\conv(\{ \rho(I), \rho(J) \})$
is an edge of $\Oc(P)$.  Then there exists a supporting hyperplane 
$\Hc$ of $\Oc(P)$ defined by the equation 
$h(x) = \sum_{i=1}^{d} a_{i}x_{i} = 1$ with each $a_{i} \in \QQ$
such that $\Hc \cap \Oc(P) = \conv(\{ \rho(I), \rho(J) \})$.
Since $\sum_{x_{i} \in I} a_{i} = \sum_{x_{j} \in J} a_{j} = 1$, one has
$\sum_{i=1}^{d} a_{i} = 2$.  In particular
$h(\rho(P)) > 1$ and $h(\emptyset) < 1$.  Thus $\Hc$ cannot be a supporting
hyperplane of $P$.  In other words, 
$\conv(\{ \rho(I), \rho(J) \})$ cannot be an edge of $P$.
Hence, if $\conv(\{ \rho(I), \rho(J) \})$ is an edge of $P$,
then either $I = \emptyset$ or $J = \emptyset$.
Let $I = \emptyset$ and $J = P$.  Suppose that
$P$ is disconnected and that $\conv(\{ \rho(\emptyset), \rho(P) \})$
is an edge of $P$.  
Again, there exists a supporting hyperplane 
$\Hc$ of $\Oc(P)$ defined by the equation 
$h(x) = \sum_{i=1}^{d} a_{i}x_{i} = 0$ with each $a_{i} \in \QQ$
such that $\Hc \cap \Oc(P) = \conv(\{ \rho(\emptyset), \rho(P) \})$. 
Let, say, $h(\rho(I)) > 0$ for those poset ideals $I$ 
with $I \neq \emptyset$ and $I \neq P$.
Since $P$ is disconnected, there exist poset ideals $I'$ and $J'$ with
$I' \cap J' = \emptyset$ and $I' \cup J' = P$.
Since $h(\rho(I')) > 0$ and $h(\rho(J')) > 0$, it follows that
$h(\rho(P)) = h(\rho(I')) + h(\rho(J')) > 0$, a contradiction.
Thus $P$ must be connected.

Conversely, suppose that $I = \emptyset$ and $J = P$ and that $P$ is connected.
Let $x_{i_{1}}, \ldots, x_{i_{q}}$ be the maximal elements of $P$
and $\Ac_{i_{j}}$ the set of those elements $y \in P$ with $y < x_{i_{j}}$.
Let $k \not\in \{ i_{1}, \ldots, i_{q} \}$.
Then we write $b_{k}$ for the number of $i_{j}$'s with $x_{k} \in \Ac_{i_{j}}$. 
Let $b_{i_{j}} = - \sharp(\Ac_{i_{j}})$.
We then claim that 
the hyperplane $\Hc$ of $\RR^{d}$ defined by the equation
$h(x) = \sum_{i=1}^{d} b_{i} x_{i} = 0$
is a supporting hyperplane of $\Oc(P)$ with
$\Hc \cap \Oc(P) = \conv(\{ \rho(\emptyset), \rho(P) \})$.
Clearly $h(\rho(P)) = h(\rho(\emptyset)) = 0$.  
Let $I$ be a poset ideal of $P$ with $I \neq \emptyset$ and $I \neq P$.
What we must prove is
$h(\rho(I)) > 0$.
To simplify the notation, suppose that
$I \cap \{ x_{i_{1}}, \ldots, x_{i_{q}} \} = \{ x_{i_{1}}, \ldots, x_{i_{r}} \}$,
where $0 \leq r < q$.  If $r = 0$, then $h(\rho(I)) > 0$.
Let $1 \leq r < q$ and $J = \cup_{j=1}^{r} (\Ac_{i_{j}} \cup x_{i_{j}})$. 
Then $J$ is a poset ideal of $P$ and $h(\rho(J)) \leq h(\rho(I))$.
We claim $h(\rho(J)) > 0$.
One has $h(\rho(J)) \geq 0$.  Moreover, $h(\rho(J)) = 0$ if and only if
no $z \in J$ belongs to $\Ac_{i_{r+1}} \cup \cdots \cup \Ac_{i_{q}}$. 
Now, since $P$ is connected, if follows that there exists 
$z \in J$ with $z \in \Ac_{i_{r+1}} \cup \cdots \cup \Ac_{i_{q}}$.
Hence $h(\rho(J)) > 0$.  Thus $h(\rho(I)) > 0$, as desired.
\end{proof}

\begin{lemma}
\label{chainedge}
Let $A$ and $B$ be antichians of $P$with $A \neq B$. 
Then $\conv(\{ \rho(A), \rho(B) \})$
forms an edge of $\Cc(P)$ if and only if 
$(A\backslash B)\cup(B\backslash A)$ 
is connected in $P$.
\end{lemma}

\begin{proof}
If $A \cup B \neq P$ and if $x_{i} \not\in A \cup B$, 
then $\conv(\{ \rho(A), \rho(B) \})$ lies in the facet $x_{i} = 0$.  
Furthermore, if $A \cup B = P$ and $A \cap B \neq \emptyset$,
then $x_{j} \in A \cap B$ is isolated in $P$ and
$x_{j}$ itself is a maximal chain of $P$.  Thus $\conv(\{ \rho(A), \rho(B) \})$ 
lies in the facet $x_{j} = 1$.  Now, 
suppose that $A \cup B = P$ and $A \cap B = \emptyset$.
Then $(A\backslash B)\cup(B\backslash A) = A \cup B = P$. 

Let $\conv(\{ \rho(A), \rho(B) \})$ be an edge of $\Cc(P)$ and
$\Hc$ a supporting hyperplane of $\Cc(P)$ defined by 
$h(x) = \sum_{i=1}^{d} a_{i}x_{i} = 1$, where each $a_{i} \in \QQ$, with
$\Hc \cap \Cc(P) = \conv(\{ \rho(A), \rho(B) \})$
and $\Cc(P) \subset \Hc^{(+)}$.
If $P$ is disconnected and if $A_{1} \cup B_{1}$ and $A_{2} \cup B_{2}$
are antichains of $P$, where $A$ is the disjoint union of $A_{1} \cup A_{2}$
and $B$ is the disjoint union of $B_{1} \cup B_{2}$, 
then $h(\rho(A_{1} \cup B_{1})) < 1$ and $h(\rho(A_{2} \cup B_{2})) < 1$.  
Hence $h(\rho(A \cup B) < 2$.  
However, since $h(\rho(A)) = 1$ and $h(\rho(B)) = 1$,
one has $h(\rho(A \cup B)) = 2$, a contradiction.  Thus 
$\conv(\{ \rho(A), \rho(B) \})$ cannot be an edge of $\Cc(P)$.
Hence $P$ must be connected if 
$\conv(\{ \rho(A), \rho(B) \})$ is an edge of $\Cc(P)$.

Now, suppose that $P$ is connected.  If there exist
$x, x' \in A$ and $y, y' \in B$
with $x < y$ and $y' < x'$, then $P$ cannot be connected.
We assume $y < x$ if $x \in A$ and $y \in B$ are comparable.
For each $x_{i} \in A$ we write $a_{i}$ for the number of elements
$y \in B$ with $y < x_{i}$.  
For each $x_{j} \in B$ we write $b_{j}$ for the number of elements
$z \in A$ with $x_{j} < z$.
Clearly $\sum_{x_{i} \in A} a_{i} = \sum_{x_{j} \in B} b_{j} = q$,
where $q$ is the number of pairs $(x, y)$ with $x \in A$, $y \in B$ and $x < y$.
Let $h(x) = \sum_{x_{i} \in A} a_{i}x_{i} + \sum_{x_{j} \in B} b_{j}x_{j}$
and $\Hc$ the hyperplane of $\RR^{d}$ defined by $h(x) = d$.
Then $h(\rho(A)) = h(\rho(B)) = q$.  We claim that, for any antichain $C$ 
of $P$ with $C \neq A$ and $C \neq B$, one has $h(\rho(C)) < q$.
Let $C = A' \cup B'$ with $A' \subset A$ and $B' \subset B$.
Since $P = A \cup B$ is connected and since $C$ is an antichain of $P$, 
it follows that
$\sum_{x_{i} \in A'}a_{i} + \sum_{x_{j} \in B'}b_{j} < q$.
Thus $h(\rho(C)) < q$, as desired.
\end{proof}

Now we ask the question whether there exists a separating hyperplane 
of an order polytope as well as that of a chain polytope.

\begin{lemma}
\label{x_a=x_b}
Let $x_{i}, x_{j} \in P$ with $x_{i} \neq x_{j}$ and $\Hc_{i,j}$ the hyperplane of
$\RR^{d}$ defined by the equation $x_{i} = x_{j}$.
Then the following conditions are equivalent:
\begin{enumerate}
\item[{\rm (i)}] $\Hc_{i,j}$ is a separating hyperplane of $\Oc(P)$;
\item[{\rm (ii)}] $\Hc_{i,j}$ intersects the interior of $\Oc(P)$;
\item[{\rm (iii)}] $x_{i}$ and $x_{j}$ are incomparable in $P$.
\end{enumerate}
\end{lemma}

\begin{proof}
The implication (i) $\Rightarrow$ (ii) is obvious. 
Suppose (ii).  Then there exist poset ideals $I$ and $J$ of $P$
with $\rho(I) \in \Hc_{i,j}^{(+)} \setminus \Hc_{i,j}$ and
$\rho(J) \in \Hc_{i,j}^{(-)} \setminus \Hc_{i,j}$.
In other words, there exist poset ideals $I$ and $J$ of $P$ 
with $x_{i}\in I \setminus J$ and $x_{j} \in J \setminus I$.
Thus in particular $x_{i}$ and $x_{j}$ are incomparable in $P$.
Hence (ii) $\Rightarrow$ (iii) follows.

Suppose (iii).  Let $I$ be the poset ideal of $P$
consisting of those $y \in P$ with $y \leq x_{i}$
and $J$ the poset ideal of $P$ 
consisting of those $y \in P$ with $y \leq x_{j}$.
Since $x_{i}$ and $x_{j}$ are incomparable in $P$,
it follows that $x_{i} \not\in J$ and
$x_{j} \not\in I$.  Thus
$\rho(I) \in \Hc_{i,j}^{(+)} \setminus \Hc_{i,j}$ and
$\rho(J) \in \Hc_{i,j}^{(-)} \setminus \Hc_{i,j}$.
Hence $\Hc_{i,j}$ intersects the interior of $\Oc(P)$.
Let, in general, $I'$ and $J'$ be poset ideals of $P$
with $\rho(I') \in \Hc_{i,j}^{(+)} \setminus \Hc$ and
$\rho(J') \in \Hc_{i,j}^{(-)} \setminus \Hc$.
In other words, $x_{i} \in I \setminus J$ and $x_{j} \in J \setminus I$.
Hence $I \not\subset J$ and $J \not\subset I$.
Lemma \ref{orderedge} then guarantees that
$\conv(\{\rho(I), \rho(J)\})$ cannot be an edge of $\Oc(P)$.
Hence
$\Hc_{i,j}$ is a separating hyperplane of
$\Oc(P)$, as desired.
\end{proof}

\begin{lemma}
\label{sephyperchain}
Let $\Hc$ be the hyperplane of $\RR^{d}$ defined by the equation
$\sum_{i=1}^{d} x_{i} - 1 = 0$.
Then the following conditions are equivalent:
\begin{enumerate}
\item[{\rm (i)}] $\Hc$ is a separating hyperplane of $\Cc(P)$;
\item[{\rm (ii)}] $\Hc$ intersects the interior of $\Cc(P)$;
\item[{\rm (iii)}] $P$ is not a chain.
\end{enumerate}
\end{lemma}

\begin{proof}
The implication (i) $\Rightarrow$ (ii) is obvious. 
Suppose (ii).  Since the origin $\rho(\emptyset)$
of $\RR^{d}$ belongs to $\Hc^{(-)} \setminus \Hc$, 
there is an antichain $A$ of $P$ with
$\rho(A) \in \Hc^{(+)} \setminus \Hc$.
Then $\sharp(A) \geq 2$.  Thus $P$ cannot be a chain. 
Hence (ii) $\Rightarrow$ (iii) follows.

Suppose (iii).  One has an antichain $A$ of $P$ with
$\sharp(A) \geq 2$.  Then $\rho(A) \in \Hc^{(+)} \setminus \Hc$
and $\rho(\emptyset) \in \Hc^{(-)} \setminus \Hc$.
Hence $\Hc$ intersects the interior of $\Cc(P)$.
Clearly $\rho(\emptyset)$ is a unique vertex of $\Cc(P)$ belonging to 
$\Hc^{(-)} \setminus \Hc$.  Let $B$ be an antichain of $P$
with $\rho(B) \in \Hc^{(+)} \setminus \Hc$.
Thus $\sharp(B) \geq 2$.  
Since $B = (\emptyset \setminus B) \cup (B \setminus \emptyset)$
is disconnected in $P$, Lemma \ref{chainedge} says that
$\conv(\{\rho(\emptyset),\rho(B)\})$ cannot be an edge.
Hence $\Hc$ is a separating hyperplane of $\Cc(P)$, as desired.
\end{proof}

By virtue of Lemmata \ref{x_a=x_b} and \ref{sephyperchain}, 
it follows immediately that

\begin{theorem}
\label{SepHypOC}
Let $P$ be a finite poset, but not a chain. 
Then each of the order polytope $\Oc(P)$ and the chain polytope $\Cc(P)$
possesses a separating hyperplane. 
\end{theorem}

\subsection{Description of seperating hyperplanes for order and chain polytopes}

In this subsection, we study the necessary and sufficient conditions such that the following hyperplane 
$$\Hc:\, h(x)=c_1x_1+c_2x_2+\cdots+c_dx_d=0$$
becomes a seperating hyperplane for a centain $d$-element poset $\Pc$. This study can be very difficult for general posets. 
 Therefore, we focus on the following three basic posets: disjoint chains; binary trees (assume connected); and zigzag posets (assume connected).
 Notice that there are no ``X'' shape in all of the three classes of posets, therefore their chain polytopes and order polytopes are unimodular 
equivalent (\cite{HL}). In this subsection, we will focus on order polytopes, and all results are also true for chain polytopes.

First, by the definitions of seperating hyperplanes, together with Lemma \ref{vertex} and Lemma \ref{orderedge} about the descriptions of the vertices and 
edges for order polytopes, we have the following description.
\begin{lemma}\label{checkcut}
$\Hc$ is a seperating hyperplane for $\Oc(P)$ if and only if the following two properties are satisfied:
\begin{enumerate}
 \item \label{nontrivial}there exist two poset ideals $I$ and $J$ such that $h(\rho(I))>0$ and $h(\rho(J))<0$ (getting two nontrivial subpolytopes);
 \item \label{nobad}$h(\rho(I))h(\rho(J))\ge 0$, for each pair of poset ideals $I$ and $J$ such that $(I\backslash J) \cup (J\backslash I)$ is
 connected in $P$. 
\end{enumerate}
\end{lemma}
We call a pair of poset ideals $I$ and $J$ that does not satisfy the second property in Lemma \ref{checkcut} a {\em bad pair} for $h$, i.e., $h(\rho(I))h(\rho(J))<0$ 
and $(I\backslash J) \cup (J\backslash I)$ is connected in $P$. 
In the rest of this subsection, we will prove most necessary conditions for being a seperating hyperplane by constructing bad pairs. We are 
looking for posets which have the following property.

Consider the following three properties of the hyperplane $\Hc$. 
\begin{property}\label{description}
Given a poset $P$, the following form the necessary and sufficient conditions for $\Hc$ to be a seperating hyperplane for $\Oc(P)$.
 \begin{enumerate}
 \item \label{min}There exist two minimal elements $i$ and $j$ such that $c_i>0$ and $c_j<0$;
 \item \label{eqab}non zero coefficients all have the same absolute value, i.e., $c_i\in\{0,1,-1\}$ after rescaling, for all $i=1,2,\dots,d$;
 \item \label{unique}coefficients for minimal elements uniquely determine the other coefficients. Here we always try to aviod having zero 
 coefficients.
\end{enumerate}
\end{property}

Notice that once Property \ref{description} is true for some poset $P$,  we can easily check whether a hyperplane is a 
seperating hyperplane for $\Oc(P)$. Moreover, the total number of seperating hyperplane
will be $2^{\#\{\text{of min elements in }P\}}$. Among the  three classes of posets we mentioned: 
disjoint chains, connected binary trees and connected 
zigzag posets, only disjoint chains satisfy Property \ref{description}. We will provide counter examples for 
the other two posets and give the best possible results under certain conditions.
\begin{proposition}\label{chain}
Property \ref{description} is true for  disjoint chains.
\end{proposition}
\begin{proof}We first prove that all three conditions listed in Property \ref{description} are necessary 
for $\Hc$ to be a seperating hyperplane.
 \begin{enumerate}
  \item By Lemma \ref{checkcut} (\ref{nontrivial}), there exists one order ideal $I$ of $\Pc$, such that $h(\rho(I))>0$. 
  We assume $I$ is connected, otherwise we look at the chain decomposition of $P=C_1\cup\cdots\cup C_r$ and consider $I\cap C_i$,
  for $i=1,\dots,r$. At least one of the intersections is nonempty and satisfies $h(\rho(I\cap C_{i}))》>0$. Now back to the case when 
  $I$ is connected. Since $I$ is a chain, there exists a unique minimal element $i$ in $I$.
  We claim that $c_i\ge0$, where $c_i$ is the coefficient of $x_i$ in $\Hc$. In fact, if $c_i<0$, $I$ and $J=\{i\}$ is a bad pair. 
  Actually, here we can assume $c_i>0$, since in the case $c_i=0$, we can simplely throw this element away from the poset 
  and look at the new minimal element in the subposet $\Pc\backslash\{i\}$. Since the whole $I$ can not have all zero coefficients,
  we will just assume $c_i\neq 0$. Similarly, we also have another minimal element $j$ with $c_j<0$.
  \item We first prove that nonzero coefficients of the minimal elements need to have the same absolute value. For example, consider
the following poset.
$$ 
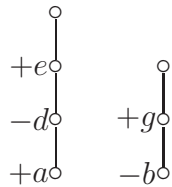
\xy 0;/r.17pc/: (0,0)*{\circ}="a";
 (0,10)*{\circ}="b";
 (0,20)*{\circ}="c";
 (0,30)*{\circ}="d";
 (20,0)*{\circ}="e";
 (20,10)*{\circ}="f";
 (20,20)*{\circ}="g";
 (-5,0)*{+a};
 (-5,10)*{-d};
 (-5,20)*{+e};
 (15,0)*{-b};
 (15,10)*{+g};
 "a"; "b"**\dir{-};
 "b"; "c"**\dir{-};
 "c"; "d"**\dir{-};
 "e"; "f"**\dir{-};
 "f"; "g"**\dir{-};
\endxy$$

Without lose of generality, pick $c_a>0$, $c_b<0$. Suppose $|c_b|>|c_a|$.
Let $I=\{a\}$, $J=\{b,a\}$. Then $(I,J)$ is a bad pair. So we need $|c_b|=|c_a|$. Consider all pairs of minimal elements with
opposite signs, we have all their coefficients have the same absolute value.

Now consider the pair $I=\{a\}$, $J=\{a,d\}$, in order to make $(I,J)$ not bad, we need $c_d\ge -c_a=c_b$. Consider the pair
$I=\{b\}$, $J=\{b,a,d\}$, we have $c_d\le 0$. Then consider the pair
$I=\{a,d\}$, $J=\{b,a,d\}$, since we want to avoid zero coefficient, assume $c_d\neq 0$, therefore we have
$c_d\le -c_a$. Therefore, we need $c_d=-c_a$. For the same reason, 
we have $c_g=c_a$.
Now consider $c_e$. Similar as above, the pair $(\{a\},\{a,d,e\})$, $(\{b\},\{b,a,d,e\})$ and $(\{a,d,e\},\{b,a,d,e\})$ implies $c_e=c_a$.
Keep going up
this way, we can show that the signs along each chain need to alternate and their coefficients have the same absolute value. 
\item We have  just shown in the previous part that given the coefficients of the minimal elements, there exists 
a unique way to extend the coefficients
to other elements (assume avoiding zero coefficients), which is exactly Property \ref{description} (\ref{unique}). 
 \end{enumerate}
Now we want to show that if a hyperplane $\Hc$ satisfies the three conditions listed in Property \ref{description}, then $\Hc$ 
is a seperating hyperplane. Condition (\ref{min}) guarantees part (\ref{nontrivial}) in Lemma \ref{checkcut}. Now we want to show 
that there is no bad pair. For any pair of poset ideals $(I,J)$, if $J\backslash I$ is connected, then $J\backslash I$ is a 
segment in a chain. By the necessary conditions on the coefficients of $\Hc$, $\sum_{i\in J\backslash I}c_i\in\{-1,0,1\}$.
As a result, no matter what is the value of $h(v_I)$, we always have $h(v_I)h(v_J)\ge 0$.

\end{proof}

\begin{proposition}\label{tree}
 For the binary trees, the following are true:
 \begin{enumerate}
  \item Property \ref{description} (\ref{min}) is necessary.
  \item Property \ref{description} (\ref{eqab}) is not necessary.
  \item Assume a separating hyperplane $\Hc$ satisfying Property \ref{description} (\ref{min} and \ref{eqab}), then (\ref{unique}) 
  is also necessary.
 \end{enumerate}
 However, all three conditions in Property \ref{description} together are not sufficient for a hyperplane to be a seperating hyperplane.
\end{proposition}
\begin{proof}
\begin{enumerate}
 \item  We want to show that,
 there exist two minimal elements $i$ and $j$ such that $c_i>0$ and $c_j<0$. The argument in the proof for 
 the disjoint union of chains also works here. The key point is that for any connected poset ideal $I$ in the binary tree 
 and one of its minimal element $i$, $I\backslash\{i\}$ is still connected in $\Pc$.
 \item The argument that all the minimal elements have the same absolute value still holds as in the disjoint union of chains. But 
 it is possible that not all elements have the same absolute value. For example. 
 consider the typerplane as the following labelled represented poset, where the label for an element $i$ in $P$ is the coefficient
 $c_i$ in $\Hc$. We can check that there are no bad pairs for $\Hc$, thus $\Hc$ is a seperating hyperplane. But not all 
 coefficients in $\Hc$ have the same absolute value.
 $$ 
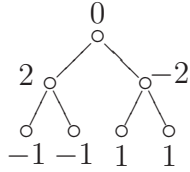
\xy 0;/r.15pc/: (0,0)*{\circ}="a";
 (10,0)*{\circ}="b";
 (5,10)*{\circ}="c";
 (-5,20)*{\circ}="d";
 (-15,10)*{\circ}="e";
 (-20,0)*{\circ}="f";
  (-10,0)*{\circ}="g";
 (0,-5)*{1};
 (10,-5)*{1};
 (-20,-5)*{-1};
 (-10,-5)*{-1};
 (-5,25)*{0};
  (10,12)*{-2};
 (-20,12)*{2};
 "a"; "c"**\dir{-};
 "b"; "c"**\dir{-};
  "e"; "f"**\dir{-};
 "e"; "g"**\dir{-};
  "d"; "e"**\dir{-};
 "d"; "c"**\dir{-};
\endxy$$

 \item Now assume all coefficients have the same absolute value, and thus can only take value from $\{-1,0,1\}$ after rescaling. So here
 we only need to talk about the sign for an element $i$ in $P$ ($+$ refers to $c_i=1$ and $-$ refers to $c_i=-1$). Now we want to 
 show that the sign of an element is determined by the sign of its two children. Here ``the sign of the child" refers to the sign of
the poset ideal generated by that child. 
 In particular, there are exactly six local sign patterns:
$$ 
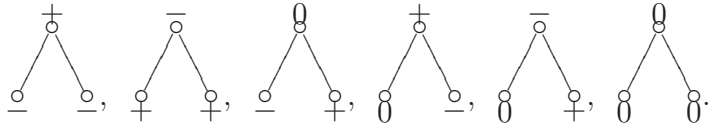
\xy 0;/r.22pc/: (0,0)*{\circ}="a";
 (10,0)*{\circ}="b";
 (5,10)*{\circ}="c";
 (0,-2)*{-};
 (10,-2)*{-};
 (5,12)*{+};
 "a"; "c"**\dir{-};
 "b"; "c"**\dir{-};
\endxy,\,\,\, \xy 0;/r.22pc/: (0,0)*{\circ}="a";
 (10,0)*{\circ}="b";
 (5,10)*{\circ}="c";
 (0,-2)*{+};
 (10,-2)*{+};
 (5,12)*{-};
 "a"; "c"**\dir{-};
 "b"; "c"**\dir{-};
\endxy,\,\,\,\xy 0;/r.22pc/: (0,0)*{\circ}="a";
 (10,0)*{\circ}="b";
 (5,10)*{\circ}="c";
 (0,-2)*{-};
 (10,-2)*{+};
 (5,12)*{0};
 "a"; "c"**\dir{-};
 "b"; "c"**\dir{-};
\endxy,\,\,\,\xy 0;/r.22pc/: (0,0)*{\circ}="a";
 (10,0)*{\circ}="b";
 (5,10)*{\circ}="c";
 (0,-2)*{0};
 (10,-2)*{-};
 (5,12)*{+};
 "a"; "c"**\dir{-};
 "b"; "c"**\dir{-};
\endxy,\,\,\, \xy 0;/r.22pc/: (0,0)*{\circ}="a";
 (10,0)*{\circ}="b";
 (5,10)*{\circ}="c";
 (0,-2)*{0};
 (10,-2)*{+};
 (5,12)*{-};
 "a"; "c"**\dir{-};
 "b"; "c"**\dir{-};
\endxy,\,\,\,\xy 0;/r.22pc/: (0,0)*{\circ}="a";
 (10,0)*{\circ}="b";
 (5,10)*{\circ}="c";
 (0,-2)*{0};
 (10,-2)*{0};
 (5,12)*{0};
 "a"; "c"**\dir{-};
 "b"; "c"**\dir{-};
\endxy.$$
Notice that $0$ appears if and only if its children have a $+$ and a $-$. 
For two elements $a,b$ with a common parent $d$,
\begin{enumerate}
\item suppose $c_b=c_a=1$. Let $e$ be a
minimal element with $c_e=-1$. Then by the pair $(I=\{e\},J=<d,e>)$ ($J$ is the poset ideal generated by $d$ and $e$), 
we have $h(\rho(J))\ge 0$,
and thus 
$c_d=-1$. This corresponds to the second tree above, and the same for the first tree.
\item suppose $c_b=-c_a>0$.  Then by the pair
$(\{b\},\{a,b,d\})$ and $(\{a\},\{a,b,d\})$, we have
$c_d=0$, which corresponds to the third tree above.
\item  suppose $c_b>0$ and $c_a=0$.  This indicates
that $a$ is larger than some minimal element $e$ with $c_e<0$. Then by the pair
$(\{e\},<d>)$ and $(\{b\},<d>)$, we have $h(\rho(<d>))=0$, thus
$c_d=-c_b$, which corresponds to the forth tree above. The fifth and the sixth tree can be obtained in a similar way.
\end{enumerate}
\item Following the above rule will not always result in a separating hyperplane. For example, 
consider the hyperplane represented by the following labelled poset.
 $$ 
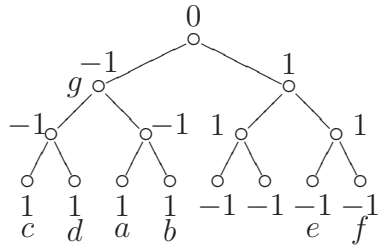
\xy 0;/r.15pc/: (0,0)*{\circ}="a";
 (10,0)*{\circ}="b";
 (5,10)*{\circ}="c";
 (-5,20)*{\circ}="d";
 (-15,10)*{\circ}="e";
 (-20,0)*{\circ}="f";
  (-10,0)*{\circ}="g";
  (-10,20)*{g};
 (0,-5)*{1};
 (10,-5)*{1};
 (-20,-5)*{1};
 (-10,-5)*{1};
 (0,-10)*{a};
 (10,-10)*{b};
 (-20,-10)*{c};
 (-10,-10)*{d};
 (-5,25)*{-1};
  (10,12)*{-1};
 (-20,12)*{-1};
 "a"; "c"**\dir{-};
 "b"; "c"**\dir{-};
  "e"; "f"**\dir{-};
 "e"; "g"**\dir{-};
  "d"; "e"**\dir{-};
 "d"; "c"**\dir{-};
 (40,0)*{\circ}="a1";
 (50,0)*{\circ}="b1";
 (45,10)*{\circ}="c1";
 (35,20)*{\circ}="d1";
 (25,10)*{\circ}="e1";
 (20,0)*{\circ}="f1";
  (30,0)*{\circ}="g1";
 (40,-5)*{-1};
 (50,-5)*{-1};
  (40,-10)*{e};
 (50,-10)*{f};
 (20,-5)*{-1};
 (30,-5)*{-1};
 (35,25)*{1};
  (50,12)*{1};
 (20,12)*{1};
 "a1"; "c1"**\dir{-};
 "b1"; "c1"**\dir{-};
  "e1"; "f1"**\dir{-};
 "e1"; "g1"**\dir{-};
  "d1"; "e1"**\dir{-};
 "d1"; "c1"**\dir{-};
   (15,30)*{\circ}="o";
   "d"; "o"**\dir{-};
 "d1"; "o"**\dir{-};
  (15,35)*{0};
\endxy$$
One can easily check that the above hyperplane follows the six local rules listed above as well the other two conditions in Property
\ref{description}. However, for example, $I=<g,e,f>$ and $J=<a,b,c,d,e,f>$ is a bad pair.
 \end{enumerate}
\end{proof}

\begin{proposition}\label{zigzag}
 For the zigzag posets, Property \ref{description} (\ref{min}) is not necessary for $\Hc$ to be a seperating hyperplane. However,
 for any hyperplane $\Hc$ with Property \ref{description} (\ref{min}), the rest two conditions listed in Property \ref{description} 
 are necessary and sufficient conditions for $\Hc$ to be a seperating hyperplane. 
\end{proposition}
\begin{proof}The following example is a seperating hyperplane but does not satisfy Property \ref{description} (\ref{min}). 
$$ \xy 0;/r.22pc/: (0,0)*{\circ}="a";
 (10,0)*{\circ}="b";
 (5,-10)*{\circ}="c";
 (0,2)*{-1};
 (10,2)*{-1};
 (7,-10)*{1};
 "a"; "c"**\dir{-};
 "b"; "c"**\dir{-};
\endxy$$

Now assume
$\Hc$ is a hyperplane satisfying Property \ref{description} (\ref{min}). We first prove that if $\Hc$ is a seperating hyperplane, then both 
 Property \ref{description} (\ref{eqab}) and (\ref{unique}) are true.
 \begin{enumerate}
  \item We want to prove that all the nonzero coefficients in any separating hyperplane for a zigzag poset have the same 
  absolute value. First notice that, all the minimal elements have the same absolute value, as proved in Proposition \ref{chain}. Following
  the same proposition, all the non maximal elements (if nonzero) have the same absolute value. As for the maximal elements, let us has
   a closer look at the zigzag poset. One maximal element $m$ covers at most two minimal elements $p,q$. For the case $m$ only covers one
   minimal element, we have the coefficient $c_m$ need to have the same absolute value for the same reason as disjoint chains proved in
   Proposition \ref{chain}. Now there are two cases when $m$ covers two minimal elements $p,q$: 
   \begin{enumerate}
    \item $c_p \cdot c_q<0$. Let $I=<m>$ be the poset ideal generated by $m$. Consider the pair $I$ and $J$, where $J=\{p\}$ or $\{q\}$.
    We have $h(\rho(I))=0$, which implies $|c_m|\le 1$.
    \item $c_p \cdot c_q>0$. Say $c_p=c_q=1$. Let $n$ be a maximal element adjacent to $m$ that covers two 
    minimal elements with different signs. For example,
    $$ 
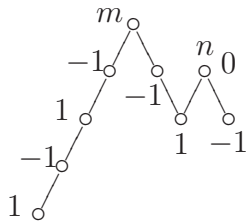
\xy 0;/r.15pc/: (0,0)*{\circ}="a";
 (10,0)*{\circ}="b";
 (5,10)*{\circ}="c";
 (0,-5)*{1};
 (10,-5)*{-1};
 (10,12)*{0};
 (5,15)*{n};
 "a"; "c"**\dir{-};
 "b"; "c"**\dir{-};
 (-5,10)*{\circ}="d";
 (-10,20)*{\circ}="e";
 (-15,10)*{\circ}="f";
  (-20,0)*{\circ}="g";
 (-25,-10)*{\circ}="h";
 (-30,-20)*{\circ}="i";
 (-8,5)*{-1};
 (-15,22)*{m};
 (-20,12)*{-1};
  (-25,2)*{1};
 (-30,-8)*{-1};
 (-35,-18)*{1};
 "a"; "d"**\dir{-};
 "d"; "e"**\dir{-};
  "e"; "f"**\dir{-};
 "g"; "f"**\dir{-};
  "g"; "h"**\dir{-};
 "i"; "h"**\dir{-};
\endxy$$ Consider the poset ideal $I=<m,n>$. Similar as the previous case, we have $h(\rho(I))=0$,
    which still implies $|c_m|\le 1$. 
   \end{enumerate}

  \item Since $\Hc$ satisfies conditions (\ref{min}) and (\ref{eqab}) in Property \ref{description}, 
  once we fix the signs of all the minimal elements, all elements except those  maximal
  are uniquely determined the same way as the disjoint chains (Proposition \ref{chain}). As for the the maximal elements, 
  they are uniquely determined
  by the signs of their two children the same as the binary trees (Proposition \ref{tree}).
  \end{enumerate}
   Now we want to prove that any hyperplane $h(x)=0$ satisfying the three conditions listed in Property \ref{description} 
   is a seperating hyperplane. The condition (\ref{min}) in Property \ref{description} implies condition (\ref{nontrivial}) in Lemma
   \ref{checkcut}. Now we want to show that there are no bad pairs. 
  Notice that by
  the rules descripted above, any connected component has value sum to $\{1,0,-1\}$. In the case $I\subset J$, if $h(\rho(I))<0$, then
  $h(\rho(J))=h(\rho(I))+h(\rho{J\backslash I})<1$, since $h(\rho(J\backslash I))<1$. Now we claim that for the zigzag poset, the condition that
  $(I\backslash J)\cup (J\backslash I)$ is connected, implies that $I\backslash J$ or $J\backslash I$ is empty. Consider a generic connected
  subposet $S=(I\backslash J)\cup (J\backslash I)$. We want to show that $S\subset I$ or $S\subset J$. 
  If $S$ only has one maximal element, then it is clear that all the elements belong to the same order ideal as the maximal
  element (either $I$ or $J$). If there are more than one maximal element, see the following example. 
     $$ 
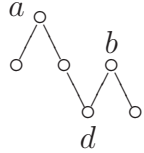
\xy 0;/r.15pc/: (0,0)*{\circ}="a";
 (10,0)*{\circ}="b";
 (5,10)*{\circ}="c";
 (5,15)*{b};
 "a"; "c"**\dir{-};
 "b"; "c"**\dir{-};
 (-5,10)*{\circ}="d";
 (-10,20)*{\circ}="e";
 (-15,10)*{\circ}="f";
 (0,-5)*{d};
 (-15,22)*{a};
 "a"; "d"**\dir{-};
 "d"; "e"**\dir{-};
  "e"; "f"**\dir{-};
\endxy$$
  Consider two adjacent maximal elements (here they are $a$ and $b$ in 
  the example). These two maximal elements cover a common minimal
  element $d$, because this subposet is connected. Then $d$ belongs to the same poset ideal as both $a$ and $b$. 
  Therefore, both $a$ and $b$ belong to the same poset ideal. This shows that $S$ belongs to either $I$ or $J$.
\end{proof}

%
%
\section{Birkhoff polytopes}
Birkhoff polytopes $B_n$ are defined to be the convex hull of all $n\times n$ nonnegative matrices with row sum and column sum 
equal to one. These matrices are known as the doubly stochastic matrices.
Here we consider an $n\times n$ matrix as a $n^2$-vector. 
Birkhoff polytopes are well-studied polytopes and have many applications, in combinatorial optimization 
and  Bayesian statistics, for example. In this section, we look for seperating hyperplanes for $B_n$ (Theorem \ref{more}).

In the rest of the section, we assume the hyperplanes have the form $h(x)=0$, but actually all the results holds
for general hyperplanes $h(x)=r$ for any constant $r$. We start with the following known properties of Birkhoff polytopes $B_n$.
Here we use both the one line notation and the cycle notation for a permutation. For example, $w=34256187$ is the one line 
notation for the permutation sending $1\rightarrow 3$, $2\rightarrow 4$, $3\rightarrow 2$, $4\rightarrow 5$, $5\rightarrow 6$, 
$6\rightarrow 1$, 
$7\rightarrow 8$ and $8\rightarrow 7$. The cycle notation for $w$ is $(132456)(78)$, thus $w$ has two cycles.
\begin{enumerate}
\item $\dim B_n=(n-1)^2$;
\item $B_n$ has $n!$ vertices, which are all the matrices corresponding to permutations $S_n$;
\item permutations $w$ and $u$ form an edge in $B_n$ if and only if $w^{-1}u$ has one cycle (excluding the fixed points).  [reference?]
\end{enumerate}
In particularly, for $n=3$, $w^{-1}u$ has one cycle for any $w,u\in S_3$. In other words, the skeleton graph for $B_3$ is the complete graph $K_6$.  Therefore, there are no seperating hyperplanes for $B_3$. Moreover, we have
\begin{lemma}\label{basic}$B_4$ has no seperating hyperplanes. 
\end{lemma}
\begin{proof}Suppose there exists a seperating hyperplane with coefficients indicated in the following matrix: 
$$\left(
  \begin{array}{cccc}
    a & b & c & d \\
    e & f & g & h \\
    i & j &  k & \ell \\
    m & n & o & p \\
  \end{array}
\right).$$
We use $x_w$ to represent the vector corresponding to the permutation matrix for a permutation $w$. By symmetry, assume $h(x_{\text{id}})>0$. The identity permutation is connected with all other permutations except for three with two cycles $(12)(34)$, $(13)(24)$ and $(14)(23)$. Then for any permutation $w$ that is not the above three, we have $h(x_w)\ge 0$, and the only possible $u$'s with $h(x_u)<0$ are among the above three. Without generality, assume $h(x_{(12)(34)})<0$. Then note that for the permutation $(12)(34)$, it is connected to all other permutations except for $\text{id}$, $(13)(24)$ and $(14)(23)$. Therefore, $h(x_v)=0$ for all permutations $v$ with one cycles.

Now notice that $h(x_{(12)(34)})+h(x_{(13)(24)})=h(x_{2143})+h(x_{3412})=(e+b+o+\ell)+(i+n+c+h)=(e+n+c+\ell)+(i+b+o+h)=h(x_{2413})+h(x_{3142})=h(x_{(1243)})+h(x_{(1342)})=0$,
therefore, $h(x_{(13)(24)})> 0$. Similarly, we can get $h(x_{(14)(23)})> 0$. But then 
$0<h(x_{(13)(24)})+h(x_{(14)(23)})=(i+n+c+h)+(m+j+g+d)=(i+n+g+d)+(m+j+c+h)=h(x_{(1324)})+h(x_{1423})=0$,
a contradiction. Therefore, there does not exist any seperating hyperplane. 
\end{proof}
\begin{remark}Even though Lemma \ref{basic} is a special case of Theorem \ref{more}, we still state it separately as a lemma, since its proof provides a good example for Theorem \ref{more}.
\end{remark}
\begin{theorem}\label{more}$B_n$ has no seperating hyperplanes.
\end{theorem}
\begin{proof}Assume there is a hyperplane $h(x)=0$. 
By symmetry, assume $h(x_{\text{id}})>0$. Since all permutations with one cycle are connected with $\text{id}$, we have $h(x_{u})\ge 0$ for all $u$ with one cycle. Suppose $h(x_{v})<0$ for some permutation $v$ with $k$ cycles. Assume $k$ is the smallest such number. In other words, $h(x_{w})\ge 0$, for all $w$ with fewer than $k$ cycles. Notice that $k>1$. First notice that $h(x_{\sigma})=0$, for all $\sigma$ connected with $v$, and have fewer cycles than $v$. In fact, since $\sigma$ has fewer than $k$ cycles, we have $h(x_{\sigma})\ge 0$. On the other hand, since $\sigma$ is connected with $v$, $h(x_{\sigma})> 0$ can not happen. Therefore, $h(x_{\sigma})=0$.

Now we apply the method in Lemma \ref{basic} to show that $h(x_{v})<0$ can not happen. Write in cycle notation $v=(C_1)(C_2)(C_3)\cdots (C_{k})$, where each $C_i$ is some sequence of numbers. Without lose of generality, assume $C_1=125A$ and $C_2=346B$, where $A$ and $B$ are sequences of numbers. First consider the permutation $\tau_1=(325A)(146B)C_3\cdots C_{k}$.  Notice that $$h(x_{v})+h(x_{\tau_1})=h(x_{\tau_2})+h(x_{\tau_3}),$$
where $\tau_2=(125A346B)C_3\cdots C_{k}$ and $\tau_3=(325A146B)C_3\cdots C_{k}$. One can check that $\tau_2$ and $\tau_3$ are both connected with $v$, in fact $\tau_2$ differs with $v$ by $(13)$ and $\tau_3$ differs with $v$ by $(24)$. Since $\tau_2$ and $\tau_3$ also have fewer than $k$ cycles, we proved earlier that $h(x_{\tau_2})=0$ and $h(x_{\tau_3})=0$. Therefore, $h(x_{\tau_1})>0$.

Now consider the permutation $\sigma_1=(C_3)\cdots (C_{k})$. Since it has fewer than $k$ cycles, we have $h(x_{\sigma_1})\ge0$.  Notice that
$$h(x_{\tau_1})+h(x_{\sigma_1})=h(x_{\sigma_2})+h(x_{\sigma_3}),$$
where $\sigma_2=(325A)C_3\cdots C_{k}$ and $\sigma_3=(146B)C_3\cdots C_{k}$. One can check that $\sigma_2$ and $\sigma_3$ are both connected with $v$. Since $\sigma_2$ and $\sigma_3$ both have fewer cycles than $v$, we have $h(x_{(\sigma_2)})=0$ and $h(x_{\sigma_3})=0$. This is a contradiction, since $h(x_{\tau_1})>0$ and $h(x_{\sigma_1})\ge 0$.
\end{proof}
%
%

%
%

\end{document}